\documentclass{article}

\usepackage{microtype}
\usepackage{graphicx}
\usepackage{subcaption}
\usepackage{booktabs} 

\usepackage{hyperref}

\usepackage[preprint]{icml2026}

\usepackage{amsmath}
\usepackage{amssymb}
\usepackage{mathtools}
\usepackage{amsthm}

\DeclareMathOperator*{\minimize}{min}

\DeclareMathOperator*{\subjectto}{subject\,to}

\DeclareMathOperator*{\softmax}{softmax}

\usepackage[capitalize,noabbrev]{cleveref}

\theoremstyle{plain}
\newtheorem{theorem}{Theorem}[section]
\newtheorem{proposition}[theorem]{Proposition}

\theoremstyle{definition}
\newtheorem{definition}[theorem]{Definition}

\theoremstyle{remark}
\newtheorem{remark}[theorem]{Remark}

\usepackage[textsize=tiny]{todonotes}

\icmltitlerunning{Exact Instance Compression for Convex Empirical Risk Minimization via Color Refinement}

\begin{document}

\twocolumn[
  \icmltitle{Exact Instance Compression for Convex Empirical Risk Minimization via Color Refinement}



  \icmlsetsymbol{equal}{*}

  \begin{icmlauthorlist}
    \icmlauthor{Bryan Zhu}{primes}
    \icmlauthor{Ziang Chen}{mit}
  \end{icmlauthorlist}

  \icmlaffiliation{primes}{Program for Research in Mathematics, Engineering, and Science (PRIMES-USA), Massachusetts Institute of Technology, Cambridge, USA}
  \icmlaffiliation{mit}{Department of Mathematics, Massachusetts Institute of Technology, Cambridge, USA}

  \icmlcorrespondingauthor{Bryan Zhu}{bryanrzhu@gmail.com}
  \icmlcorrespondingauthor{Ziang Chen}{ziang@mit.edu}

  \icmlkeywords{color refinement, empirical risk minimization, instance compression}

  \vskip 0.3in
]



\printAffiliationsAndNotice{}  

\begin{abstract}
  Empirical risk minimization (ERM) can be computationally expensive, with standard solvers scaling poorly even in the convex setting. We propose a novel lossless compression framework for convex ERM based on color refinement, extending prior work from linear programs and convex quadratic programs to a broad class of differentiable convex optimization problems. We develop concrete algorithms for a range of models, including linear and polynomial regression, binary and multiclass logistic regression, regression with elastic-net regularization, and kernel methods such as kernel ridge regression and kernel logistic regression. Numerical experiments on representative datasets demonstrate the effectiveness of the proposed approach.
\end{abstract}

\section{Introduction}

Empirical risk minimization (ERM) is a foundational paradigm in modern machine learning and statistical inference, underlying a wide range of models, including linear and generalized linear models, kernel methods, and regularized regression.

Although many ERM formulations are convex, solving them at scale remains computationally challenging: the complexity of first- and second-order optimization methods typically grows at least linearly with the number of samples and features, and can become significantly worse in the presence of dense data matrices, kernelized representations, or structured regularization.

These computational bottlenecks have motivated extensive research on preprocessing and problem-reduction techniques that seek to reduce instance size while preserving optimal solutions, including randomized sketching \citep{mahoney2011randomized, pilanci2016iterative}, subsampling \citep{drineas2012fast}, coreset construction \citep{feldman2020turning}, and symmetry-based reductions \citep{ravanbakhsh2017equivariance}.

In this paper, we study lossless compression of convex ERM problems: deterministic reductions that provably preserve optimal solutions and convexity guarantees while substantially reducing problem dimension. Unlike approximate methods such as sketching or coresets, lossless compression yields a reduced optimization problem whose solution can be lifted exactly to a solution of the original problem. This property is particularly attractive in settings where accuracy guarantees, interpretability, or certification are required.

Our approach builds on recent advances in dimension reduction for optimization based on color refinement, a classical algorithm from graph isomorphism theory \citep{weisfeiler1968reduction}. Color refinement iteratively assigns colors to vertices of a graph in a way that captures structural equivalence, and has been widely used as a practical heuristic for detecting symmetries in large combinatorial structures \citep{babai2016graph}. Since matrices can be naturally represented as edge-weighted bipartite graphs, color refinement can be used to identify symmetric or functionally equivalent rows and columns in optimization problems.

Recent work has shown that this idea leads to exact dimension reduction for linear programs (LPs), where variables and constraints can be grouped into equivalence classes and compressed without changing feasibility or optimality \cite{grohe}. This framework was subsequently extended to convex quadratic programs (QPs), enabling lossless reduction for objectives with quadratic structure while preserving convexity \cite{mladenov}. These results demonstrate that reductions can go beyond purely combinatorial optimization and apply to continuous convex problems.

However, existing methods are limited in scope: they rely heavily on linear or quadratic structure and do not directly apply to general convex ERM objectives, such as those arising from logistic loss, multinomial loss, elastic-net regularization, or kernelized models. Moreover, prior symmetry-based approaches typically require explicit permutation invariance of the problem, which is often absent in practical machine learning formulations even when many samples or features behave identically.

In this work, we develop a general dimension reduction theorem for differentiable convex programming and instantiate it in practical algorithms for convex ERM families. Our method identifies conditions under which samples, features, and constraints behave identically, even when the optimization problem is not explicitly invariant under permutations of these identically-behaving items. As a consequence, our framework strictly generalizes prior reduction techniques and applies to a much broader class of ERM formulations.

The remainder of the paper is organized as follows. Section~\ref{sec:preliminaries} introduces background material, notation, and relevant prior work. Section~\ref{sec:reductiontheorem} presents our generalization of color-refinement–based reductions to convex programs and shows that our method always yields at least as much reduction as approaches based on permutation invariance. Section~\ref{sec:ml} discusses applications of our theorem to machine learning models, including linear and polynomial regression, binary and multiclass logistic regression, regression with elastic-net regularization, and kernel-based models such as kernel ridge regression and kernel logistic regression. Section~\ref{sec:experiments} reports experimental results for binary logistic regression on various datasets from OpenML \cite{openml} and LIBSVM \cite{libsvm}. Section~\ref{sec:conclusion} concludes the paper.

\section{Preliminaries}\label{sec:preliminaries}

In this section, we review existing theoretical tools underlying our dimension reduction framework.

We first recall how equitable partitions lead to exact reductions for linear programs \citep{grohe} and convex quadratic programs \citep{mladenov}. Then we introduce color refinement \citep{berkholz}, a combinatorial algorithm that efficiently computes the coarsest equitable partition and serves as the algorithmic backbone of related dimension reduction approaches.

\subsection{Dimension Reduction for Linear Programs}
We begin with linear programs (LPs), where color refinement-based dimension reduction is most transparent.
Throughout, consider an LP of the form
\vspace{-1mm}
\begin{equation*}
\min \; c^\top x
\quad \text{subject to } Ax \le b,\quad l \le x \le u .
\vspace{-1mm}
\end{equation*}
We refer to the constraints \(l \le x \le u\) as the \emph{box domain}. A point \(x\) is said to be \emph{feasible} if it satisfies both the box constraints and the linear inequalities \(Ax \le b\).

The key idea is that when the LP exhibits sufficient regularity across variables and constraints, many of them can be aggregated without loss of optimality. This aggregation is formalized through the notion of equitable partitions.

We say that a \emph{partition} of a set $S$ is a set of nonempty, pairwise disjoint subsets of $S$ whose union equals $S$. Let \(\mathcal{P}\) be a partition of the row indices of \(A\), and let \(\mathcal{Q}\) be a partition of its column indices. We say that the pair \((\mathcal{P},\mathcal{Q})\) is \emph{equitable} if, for every \(S \in \mathcal{P}\) and \(T \in \mathcal{Q}\),
\vspace{-3mm}
\begin{itemize}
    \item $\sum_{j \in T} A_{ij}$ is constant over all $i \in S$, and
    \item $\sum_{i \in S} A_{ij}$ is constant over all $j \in T$.
    \vspace{-3mm}
\end{itemize}
Equivalently, each submatrix induced by \(\mathcal{P} \times \mathcal{Q}\) has both constant row sums and constant column sums.

To express aggregation algebraically, for any partition \(\mathcal{P}\) of a finite set \(V\), define the \emph{partition matrix} \(\Pi_{\mathcal{P}} \in \{0,1\}^{V \times \mathcal{P}}\) by
\vspace{-1mm}
\begin{equation*}
(\Pi_{\mathcal{P}})_{vS} =
\begin{cases}
1, & v \in S,\\
0, & \text{otherwise}.
\end{cases}
\vspace{-2mm}
\end{equation*}
We also define the scaled transpose \(\Pi_{\mathcal{P}}^{\text{Scaled}}\) by normalizing each row of \(\Pi_{\mathcal{P}}^\top\)
so that it is stochastic.

We now state the conditions under which an LP admits an exact reduced formulation. Assume that
\vspace{-3mm}
\begin{itemize}
\item for all \(T \in \mathcal{Q}\), the triples \((c_j, l_j, u_j)\) are identical for all \(j \in T\);
\vspace{-2mm}
\item the pair \((\mathcal{P},\mathcal{Q})\) is equitable with respect to \(A\);
\vspace{-2mm}
\item for all \(S \in \mathcal{P}\), the values \(b_i\) are identical for all \(i \in S\).
\vspace{-3mm}
\end{itemize}
Under these assumptions, define the reduced quantities
\vspace{-2mm}
\begin{align*}
c' &= \Pi_\mathcal{Q}^\top c, &
A' &= \Pi_\mathcal{P}^{\text{Scaled}} A \Pi_Q, &
b' &= \Pi_\mathcal{P}^{\text{Scaled}} b, \\
x' &= \Pi_\mathcal{Q}^{\text{Scaled}} x, &
l' &= \Pi_\mathcal{Q}^{\text{Scaled}} l, &
u' &= \Pi_\mathcal{Q}^{\text{Scaled}} u .
\vspace{-5mm}
\end{align*}
The reduced LP is obtained by replacing \((c,A,b,l,u)\) with
\((c',A',b',l',u')\).
It is established in \citet{grohe} that \(x\) is an optimal solution of the original LP if and only if
\(x'\) is an optimal solution of the reduced LP.
Intuitively, this procedure collapses all variables within each color \(T \in Q\)
into a single representative variable and aggregates all constraints within
each color \(S \in P\) into a single constraint.

\subsection{Dimension Reduction for Quadratic Programs}\label{subsec:qps}

We next consider convex quadratic programs (QPs), which arise naturally in many optimization and learning problems.
Specifically, consider a QP of the form
\vspace{-1.5mm}
\begin{equation*}
\min \; \tfrac{1}{2} x^\top Q x + c^\top x
\quad \text{subject to } Ax \le b,\quad l \le x \le u,
\vspace{-1.5mm}
\end{equation*}
where \(Q \succeq 0\) and is symmetric.
The reduction procedure largely mirrors that of linear programs, with one
additional requirement ensuring compatibility with the quadratic term.

In addition to the assumptions in the LP case, we require that the partition \((\mathcal{Q}, \mathcal{Q})\)
be equitable with respect to the matrix \(Q\). (Here, $\mathcal{Q}$ partitions both the rows and columns of \(Q\).)
Under this condition, the reduced quadratic matrix is given by
\begin{equation*}
Q' = \Pi_\mathcal{Q}^\top Q \Pi_\mathcal{Q} .
\end{equation*}
As shown in \citet{mladenov}, this construction preserves convexity and optimality:
solutions of the reduced QP correspond exactly to solutions of the original problem.

\subsection{Color Refinement}

Color refinement is a graph isomorphism algorithm \cite{berkholz}. It gradually refines a coloring of a graph, which can also be seen as a partition of its vertices. A coloring $\mathcal{C}_1$ \textit{refines} another coloring $\mathcal{C}_2$ if, for all $C_1\in\mathcal{C}_1$, there exists some $C_2\in\mathcal{C}_2$ such that $C_1\subseteq C_2$. Here, $C_1$ is \textit{finer}, and $C_2$ is \textit{coarser}. Additionally, we call $\mathcal{C}$ the \textit{unit coloring} if $|\mathcal{C}|=1$.

Color refinement can be extended to matrices to find the equitable partition $(\mathcal{P}, \mathcal{Q})$ that minimizes $|\mathcal{P}|$ and $|\mathcal{Q}|$, which can be shown to be unique \cite{grohe}. This is called the \textit{coarsest equitable partition}. We will present an explicit version in Section~\ref{sec:colorrefinement}.


\section{Dimension Reduction}\label{sec:reductiontheorem}
Section~\ref{sec:preliminaries} reviewed how equitable partitions yield exact reductions for linear and quadratic programs.
We now generalize this principle to differentiable convex programs, establishing a unified reduction theorem that applies to a broad class of empirical risk minimization (ERM) problems.
Our result formalizes when variables and constraints can be aggregated without requiring explicit permutation symmetry of the optimization problem.

\subsection{Reduction Theorem}

We consider a general convex program with $n$ variables and $m$ constraints:
\vspace{-2mm}
\begin{align*}
    \minimize~~&\ F(x),\\
    \subjectto~~&\ G_i(x) \le b_i, \quad i = 1,\dots,m,\quad l \le x \le u.
    \vspace{-2mm}
\end{align*}
We assume that $F$ and each $G_i$ are convex and differentiable on the box
domain $\{x \mid l \le x \le u\}$.
As in the linear and quadratic cases, our goal is to identify groups of variables and
constraints that behave identically under aggregation.
This is captured by the following notion.

\begin{definition}[Reduction Coloring]
    A coloring $\mathcal{P}$ of the constraints and a coloring $\mathcal{Q}$ of the variables of a convex program is a \textit{reduction coloring} if the following conditions hold for all $\hat{x}$ in the box domain satisfying $\hat{x}_{j_1} = \hat{x}_{j_2}$ whenever $j_1$ and $j_2$ belong to the same color in $\mathcal{Q}$:
    \vspace{-3mm}
    \begin{itemize}
        \item If $T\in\mathcal{Q}$, then for all $j_1, j_2\in T$, 
        \vspace{-1mm}
        \[\left.\frac{\partial F}{\partial x_{j_1}}\right|_{x=\hat{x}}=\left.\frac{\partial F}{\partial x_{j_2}}\right|_{x=\hat{x}}.
        \vspace{-1mm}
        \]
        \item If $S\in\mathcal{P}$ and $T\in\mathcal{Q}$, then for all $j_1, j_2\in T$, \[\left.\frac{\partial}{\partial x_{j_1}}\left(\sum_{i\in S}G_i\right)\right|_{x=\hat{x}}=\left.\frac{\partial}{\partial x_{j_2}}\left(\sum_{i\in S}G_i\right)\right|_{x=\hat{x}}.\]
        \item For each $S\in\mathcal{P}$, $G_i(\hat{x})$ is equal for every $i\in S$.
        \vspace{-2mm}
        \item For each $S\in\mathcal{P}$, $b_i$ is equal for every $i\in S$.
        \vspace{-2mm}
        \item For each $T\in\mathcal{Q}$, $(l_j, u_j)$ is equal for every $j\in T$.
        \vspace{-3mm}
    \end{itemize}
\end{definition}
Intuitively, a reduction coloring ensures that variables in the same color are indistinguishable from the perspective of the objective and aggregated constraints, and that constraints in the same color impose identical restrictions on color-constant solutions.

In general convex programs, verifying these conditions directly may be nontrivial.
In Section~\ref{sec:ml}, we provide constructive sufficient conditions for common ERM
families, where reduction colorings are computed efficiently via color refinement
applied to the data matrix, labels, and weights.

\begin{definition}
    Given a reduction coloring $(\mathcal{P}, \mathcal{Q})$ of a convex program, we define a \textit{reduced program}:
    \vspace{-3mm}
    \begin{itemize}
        \item The reduced objective $F'$ is obtained by substituting a single reduced variable $x'_T$ for all original variables $x_j$ belonging to the same $T \in \mathcal{Q}$. In other words, for each reduced vector $x'\in\mathbb{R}^{|\mathcal{Q}|}$, we set a point $\hat{x}\in\mathbb{R}^n$ such that $\hat{x}_j=x'_T$ whenever $j\in T$. Then we set $F'(x') = F(\hat{x})$.
        \vspace{-2mm}
        \item Reduced constraints $G_S'$ are given by transforming $G_i$ similarly as the objective for any $i\in S$ (well defined since $G_S'(x')=G_i(\hat x)$ is equal for all $i\in S$ for all color-constant $\hat{x}$).
        \vspace{-2mm}
        \item For $x$, $l$, and $u$, we left-multiply by $\Pi_\mathcal{Q}^{\text{Scaled}}$ to yield the reduced vectors.
        \vspace{-2mm}
        \item For $b$, we use $b'=\Pi_\mathcal{P}^{\text{Scaled}}b$.
        \vspace{-3mm}
    \end{itemize}
\end{definition}

We can now state the main reduction theorem.

\begin{theorem}\label{reductiontheorem}
    Consider any convex program with $F$ and each $G_i$ differentiable on the box domain. Let $(\mathcal{P}, \mathcal{Q})$ be a reduction coloring. If $x$ is an optimum for the original program, then $x'=\Pi_\mathcal{Q}^{\text{Scaled}}x$ is an optimum for the reduced program. If $x'$ is an optimum for the reduced program, then $x=\Pi_\mathcal{Q}x'$ is an optimum for the original program.
\end{theorem}
We briefly outline a proof of Theorem~\ref{reductiontheorem}. Suppose that $(\mathcal{P}, \mathcal{Q})$ is a reduction coloring. Let $\hat{x}=\Pi_\mathcal{Q}\Pi_\mathcal{Q}^{\text{Scaled}} x$. Note that if $j_1\in T$ and $T\in\mathcal{Q}$, then 
\vspace{-2mm}
\[\hat{x}_{j_1}=\frac{\sum_{j_2\in T}x_{j_2}}{|T|}.
\vspace{-2mm}
\] We have the following:
\vspace{-3mm}
\begin{enumerate}
    \item The objective function does not increase when $x$ is changed to $\hat{x}$. This can be shown by fixing $x$, considering the problem of minimizing $F(z)$ over $z\in\mathbb{R}^n$ subject to $\sum_{j\in T}z_j=\sum_{j\in T}x_j$ for all $T\in\mathcal{Q}$, and taking the Lagrangian at $z=\hat{x}$.
    \vspace{-2mm}
    \item If $x$ is feasible, $\hat{x}$ must be too. This is true by applying the above logic on $F$ to $\sum_{i\in S}G_i$ for each $S\in\mathcal{P}$, then using the requirement that $G_i(\hat{x})$ is constant across $i\in S$.
    \vspace{-2mm}
    \item The above statements imply equivalence of the original and reduced problems.
    \vspace{-3mm}
\end{enumerate}
A full proof is given in Appendix~\ref{app:reductiontheoremproof}.


\subsection{Beyond Permutation Symmetry}
Reduction colorings capture equivalence under aggregated interactions, rather
than requiring explicit permutation invariance of the optimization problem.
As a result, our framework can yield strictly stronger compression than approaches
based solely on symmetry groups.

Formally, let $\Gamma$ denote the group of automorphisms $(\pi, \sigma) \in S_n \times S_m$
such that for all $x$ in the box domain:
\vspace{-3mm}
\begin{itemize}
    \item $F(\pi(x)) = F(x)$;
    \vspace{-2mm}
    \item $\pi(l) = l$ and $\pi(u) = u$;
    \vspace{-2mm}
    \item $G_{\sigma(i)}(\pi(x)) = G_i(x)$ and $b_{\sigma(i)} = b_i$
    for all $i = 1,\dots,m$.
    \vspace{-3mm}
\end{itemize}
The group $\Gamma$ acts on variables by $\pi$ and on constraints by $\sigma$.
Let $\mathcal{Q}_\Gamma$ and $\mathcal{P}_\Gamma$ denote the corresponding orbit
partitions of variables and constraints.

\begin{theorem}\label{coarsenesstheorem}
    If $(\mathcal{P}, \mathcal{Q})$ is the coarsest reduction coloring, then it is at least as coarse as $(\mathcal{P}_\Gamma, \mathcal{Q}_\Gamma)$.
\end{theorem}

The proof, given in Appendix~\ref{app:coarsenesstheoremproof}, shows that any permutation symmetry induces a valid reduction coloring, implying that reduction colorings strictly generalize symmetry-based reductions.
We also provide an example where nontrivial compression is possible even in the absence of any nontrivial permutation symmetry.

\subsection{Color Refinement for Matrices}\label{sec:colorrefinement}
While matrix equitability does not explicitly appear in Theorem~\ref{reductiontheorem}, equitability is often required on the data or kernel matrix for convex ERM instantiations of our compression method. Below, we present the algorithm we use for refining a coloring $(\mathcal{P}, \mathcal{Q})$ of a matrix $A$ until it converges to the coarsest equitable partition. This is essentially a weighted case of color refinement for graphs \cite{berkholz}.
\begin{algorithm}[htb!]
\caption{Color refinement for coarsest equitable partition of a matrix}\label{alg:cep}
\begin{algorithmic}[1]
\STATE \textbf{Initialize:} \(\mathcal{P}\) and \(\mathcal{Q}\) as needed, depending on the specific case
\STATE Add all colors in \(\mathcal{P} \cup \mathcal{Q}\) to a stack \(S_\text{refine}\)

\WHILE{$S_\text{refine} \neq \emptyset$ \textbf{and} ($|\mathcal{P}| < m$ \textbf{or} $|\mathcal{Q}| < n$)}
    \STATE $R \gets \textsc{Pop}(S_\text{refine})$
    \IF{$R \in \mathcal{P}$}
        \FORALL{$T \in \mathcal{Q}$}
            \STATE Refine \(\mathcal{Q}\): partition all $j\in T$ by \(\left\{\sum_{i\in R}A_{ij}\right\}\).
        \ENDFOR
        \STATE Push all newly created colors into \(S_\text{refine}\)
    \ELSIF{$R \in \mathcal{Q}$}
        \FORALL{\(S \in\mathcal{P}\)}
            \STATE Refine \(\mathcal{P}\): partition all $i\in S$ by \(\left\{\sum_{j\in R}A_{ij}\right\}\).
        \ENDFOR
        \STATE Push all newly created colors into \(S_\text{refine}\)
    \ENDIF
\ENDWHILE
\STATE Return $(\mathcal{P}, \mathcal{Q})$
\end{algorithmic}
\end{algorithm}

When we split a color $R$ into sub-colors $\mathcal{S}=\{\gamma, \gamma', \gamma'', \dots\}$ (we assume without loss of generality that $|\gamma|\geq|\gamma'|\geq|\gamma''|\geq\dots$), $\gamma$ keeps the $R$ label, while each of $\mathcal{S}\setminus\{\gamma\}$ is assigned a distinct new label. Hence, all sub-colors in $\mathcal{S}\setminus\{\gamma\}$ are pushed into $S_\text{refine}$.

We then analyze the complexity of Algorithm~\ref{alg:cep}.
The first time the color $R$ of some index $1\leq i\leq m$ or $1\leq j\leq n$ is pushed into $S_\text{refine}$, it has size $|R|$. If the color of $i$ or $j$ ever gets pushed into $S_\text{refine}$ again, then at some point, $R$ must have been split, with $i$ or $j$ being assigned to some color in $\mathcal{S}\setminus\{\gamma\}$. Hence, the current color of $i$ or $j$ has size at most $\frac{|R|}{2}$. This implies that each $i$ can be pushed into $S_\text{refine}$ at most $\log_2(m)$ times, and each $j$ can be pushed at most $\log_2(n)$ times. Therefore, 
\vspace{-2mm}
\[\sum_{R\in\mathcal Q\text{ popped}} |R|=O(n\log (n)),
\vspace{-2mm}
\] and 
\[\sum_{R\in\mathcal P\text{ popped}} |R|=O(m\log (m)).
\vspace{-1mm}
\]
When we pop a color $R\in\mathcal{P}$ and process it, we must examine all entries of $A$ whose first index lies in $R$, which is $|R|n$ entries. Hence, the total work over all colors is 
\vspace{-2mm}
\[O\!\left(n\sum_{R\in\mathcal{P}\text{ popped}}|R|\right)=O(mn\log (m)).
\vspace{-2mm}
\] 
Similarly, the total work over all colors in $\mathcal{Q}$ is $O(mn\log (n))$. Combining these, we obtain that:

\begin{proposition}
The complexity of Algorithm~\ref{alg:cep} is 
\vspace{-2mm}
\[O(mn(\log (m)+\log (n))).
\vspace{-3mm}
\]
\end{proposition}

In many ERM settings, $A$ is sparse, meaning most entries are exactly 0 and $A$ is stored by listing only its nonzero entries. In this case, refinement can be implemented by aggregating only nonzero entries with an index in the refining color, so the runtime depends on $\text{nnz}(A)$, the number of nonzero entries in $A$, rather than $mn$.

\section{Applications in Machine Learning}\label{sec:ml}

We now instantiate the general reduction framework from Section~\ref{sec:reductiontheorem}
for several widely used convex ERM models.
In each case, we derive explicit, verifiable conditions under which samples, features,
or coefficients can be aggregated without loss of optimality.
These conditions naturally translate into practical reduction algorithms based on
color refinement applied to the data matrix, labels, and sample weights.

Throughout this section, we emphasize that the resulting reductions are
exact: solutions of the reduced problem lift to optimal solutions of the
original ERM objective.

\subsection{Least Squares Linear Regression and Polynomial Regression}\label{sec:linreg}

We begin with least squares linear regression, which serves as a canonical example
illustrating how the abstract conditions of Section~\ref{sec:reductiontheorem}
translate into concrete algebraic constraints on the data.

Consider the following least squares linear regression problem with $n$ samples and $D$ features: 
\vspace{-2mm}
\[\minimize \ F(w, b) = \|\hat{y}-y\|_2^2,
\vspace{-2mm}
\] where $\hat{y}=Xw+b1_n$ is the prediction ($1_n\in\mathbb{R}^n$ is a vector of ones, so adding $b1_n$ adds the bias term $b\in\mathbb{R}$ to each row/sample of $Xw$), $X\in\mathbb{R}^{n\times D}$ contains the input samples, and $y\in\mathbb{R}^n$ contains the ground truth labels.

\begin{theorem}\label{linregreduction}
A coloring $\mathcal{Q}$ of the coefficients $w$ is a reduction coloring for linear regression if the following conditions are met:
\vspace{-3mm}
\begin{itemize}
    \item Define $\mathcal{P}$ as the coloring given by partitioning all $1\leq i\leq n$ by 
    \vspace{-2mm}
    \[\left\{\left(T, \sum_{j\in T}X_{ij}\right)\;\middle|\; T\in\mathcal{Q}\right\}.
    \vspace{-2mm}
    \] We require that $(\mathcal{P}, \mathcal{Q})$ is equitable on $X$.
    \vspace{-2mm}
    \item If $j_1$ and $j_2$ share a color in $\mathcal{Q}$, then 
    \vspace{-2mm}
    \[\sum_{i=1}^n X_{ij_1}y_i = \sum_{i=1}^n X_{ij_2}y_i.
    \vspace{-2mm}
    \] In other words, $(X^\top y)_{j_1}=(X^\top y)_{j_2}$.
    \vspace{-3mm}
\end{itemize}
\end{theorem}
The proof is given in Appendix~\ref{app:linregproof}. We also verify that the closed-form solution of the reduced problem lifts exactly to a solution of the original regression.

In practice, these conditions can be enforced by initializing $\mathcal{P}$ as the unit coloring and $\mathcal{Q}$ by splitting features $j$ according to $\sum_{i=1}^n X_{ij}y_i$, then computing the coarsest equitable partition of $X$ refining $(\mathcal{P},\mathcal{Q})$.
The reduced regression problem is obtained by projecting
$X' = X \Pi_{\mathcal{Q}}$.

In fact, since each sample $i$ sharing a color in $\mathcal{P}$ has identical features after this operation, we can further merge and weigh these samples by projecting $X'=\Pi_\mathcal{P}^{\text{Scaled}}X\Pi_\mathcal{Q}$ and $y'=\Pi^{\text{Scaled}}_\mathcal{P}y$, with the weights matrix $W'=\Pi_\mathcal{P}^\top\Pi_\mathcal{P}$ being given by $W'_{SS}=|S|$ for all $S\in\mathcal{P}$. This is because having $|S|$ identical samples $x_S$ with possibly different $y$ is equivalent to replacing all $y$ in all such samples with their mean $y'_S$ because the loss function that we are aiming to minimize differs by a constant.

\begin{remark}
    This method can also be applied to polynomial regression and, more generally, to any fixed feature map $\phi(x)$ because these are simply linear regression with an expanded $X$. The coloring conditions imposed on the expanded data matrix are exactly the same as highlighted above. Moreover, our method also applies to weighted linear regression. This is because weighted linear regression is equivalent to transforming the feature matrix $X\mapsto W^{1/2}X$, where $W$ is the diagonal weights matrix, and transforming the ground truth $y\mapsto W^{1/2}y$.
\end{remark}

\subsection{Binary Logistic Regression}
We next consider the weighted binary logistic regression problem
\vspace{-3mm}
\begin{align*}
    \minimize &\ F(w, b) =\\
    &-\sum_{i=1}^n v_i(y_i\log (\hat{y}_i)+(1-y_i)\log (1-\hat{y}_i)),
    \vspace{-4mm}
\end{align*}
where $v\in\mathbb{R}^n$ is the weights vector, $\hat{y}=\sigma(Xw+b1_n)$ is the predicted probability vector, $\sigma$ represents the sigmoid function, and $y\in\{0,1\}^n$ are the ground truth labels.

\begin{theorem}\label{blrreduction}
A coloring $\mathcal{Q}$ of the coefficients $w$ is a reduction coloring for binary logistic regression if the following conditions are met:
\vspace{-3mm}
\begin{itemize}
    \item Define $\mathcal{P}$ as we did in Theorem~\ref{linregreduction}. We require that $(\mathcal{P}, \mathcal{Q})$ is equitable on $X$.
    \vspace{-2mm}
    \item If $j_1$ and $j_2$ share a color in $\mathcal{Q}$, then 
    \vspace{-2mm}
    \[\sum_{i=1}^n v_iX_{ij_1}y_i = \sum_{i=1}^n v_iX_{ij_2}y_i.
    \vspace{-3mm}
    \]
    \vspace{-2mm}
    \item $v_i$ is constant across all $i\in S$ for any $S\in\mathcal{P}$.
    \vspace{-3mm}
\end{itemize}
\end{theorem}
The proof of Theorem~\ref{blrreduction} is given in Appendix~\ref{app:blrproof}.

Similar to linear regression, we can satisfy all these conditions by initializing $\mathcal{P}$ to split $1\leq i\leq n$ by $v_i$ and $\mathcal{Q}$ to split $1\leq j\leq D$ by $\sum_{i=1}^n v_iX_{ij}y_i$, then finding the coarsest equitable partition of $X$ that refines $(\mathcal{P}, \mathcal{Q})$.

The reduction process for logistic regression is identical to that of linear regression. In $F'$, the summation is over $S\in\mathcal{P}$ instead of $1\leq i\leq n$. For the weights vector, we use $v'=\Pi_\mathcal{P}^\top v$. If we would like to have a binary ground truth, we can alternatively merge samples in each $S\in\mathcal{P}$ into up to two reduced weighted samples, one with label 1 and one with label 0.

\subsection{Multiclass Logistic Regression}
We now extend the reduction to multiclass logistic regression with $K$ classes:
\vspace{-2mm}
\[\minimize \ F(W, b) = -\sum_{i=1}^n v_i\log (\hat{y}_{iy_i}),
\vspace{-2mm}
\] where $W\in\mathbb{R}^{D\times K}$, $b\in\mathbb{R}^K$, $\hat{y}=\softmax(XW+1_nb^\top)$ (where softmax is applied along the class dimension) is the predicted probability matrix, and $y\in\{1, 2, \dots, K\}^n$ are the ground truth labels.
As with before, let $\mathcal{P}$ be a partition of the samples, and let $\mathcal{Q}$ be a partition of the features. We do not reduce the number of classes.
\begin{theorem}\label{mclrreduction}
A reduction coloring for multiclass logistic regression is defined the same way as for binary logistic regression in Theorem~\ref{blrreduction} (including equitability on $X$ and the requirement that $v_i$ is constant across any $S\in\mathcal{P}$). The only difference is that we generalize our requirement of equal $\sum_{i=1}^n v_iX_{ij}y_i$ across $j\in T$ for any $T\in\mathcal{Q}$ by requiring equal 
\vspace{-2mm}
\[\sum_{i=1}^nv_iX_{ij}\cdot 1\{y_i=c\}
\vspace{-2mm}
\] for all $1\leq c\leq K$.
\end{theorem}
We prove Theorem~\ref{mclrreduction} in Appendix~\ref{app:mclrproof}.

Within each class, the reduction process for multiclass logistic regression is exactly identical to that of binary logistic regression. This means that $W'=\Pi_\mathcal{Q}^{\text{Scaled}}W$, $b'=b$, $v'=\Pi_\mathcal{P}^\top v$, $\hat{y}'=\softmax(X'W'+1_{|\mathcal{P}|}b^\top)$, $X'=\Pi_\mathcal{P}^{\text{Scaled}}X\Pi_\mathcal{Q}$, $y'\in\mathbb{R}^{|\mathcal{P}|\times K}$ is a probability matrix with \[y'_{Sc}=\frac{\sum_{i\in S}1\{y_i=c\}}{|S|},
\vspace{-2mm}
\] and 
\vspace{-2mm}
\[F'(W', b')=-\sum_{S\in\mathcal{P}}v'_S\sum_{c=1}^Ky'_{Sc}\log (\hat{y}'_{Sc}).
\vspace{-2mm}
\] Alternatively, to ensure a binary ground truth, we can merge samples in each $S\in\mathcal{P}$ into up to $K$ reduced weighted samples, one per class.

\subsection{Elastic-Net Regularization}
Consider adding an elastic-net regularization term, 
\vspace{-2mm}
\[\lambda_2\|w\|_2^2+\lambda_1\|w\|_1,
\vspace{-2mm}
\] to the objective function of any case of convex ERM. Theorem~\ref{reductiontheorem} does not apply directly here because the L1 regularization term is not always differentiable. However, we note that the definition of a reduction coloring does not change since this regularization term cannot increase if $w$ is replaced with $\hat{w}$, regardless of $\mathcal{Q}$, by Jensen's inequality. For example, we verify that the closed-form solution for ridge regression is preserved in Appendix~\ref{app:enetproof}.

\subsection{Note on StandardScaler}\label{sec:standardscaler}

For models such as linear and logistic regression, it is often beneficial to apply StandardScaler from scikit-learn \cite{scikit-learn} to improve numerical stability and conditioning. Given a data matrix \(X \in \mathbb{R}^{n \times D}\), StandardScaler computes feature-wise empirical means \(\mu \in \mathbb{R}^D\) and standard deviations \(\sigma \in \mathbb{R}^D\), and transforms each sample \(x \in \mathbb{R}^D\) as 
\vspace{-1mm}
\[x_{\mathrm{sc}} = D_{\sigma}^{-1}(x - \mu),
\vspace{-1mm}
\] where \(D_{\sigma} = \mathrm{diag}(\sigma_1,\dots,\sigma_D)\).

However, we empirically observe that applying StandardScaler prior to computing the coarsest reduction coloring typically produces a strictly finer partition. In other words, the equitable partition of the standardized matrix is usually a refinement of the equitable partition of the original matrix. To preserve equivalence while retaining the numerical benefits of standardization, we proceed as follows:
\vspace{-3mm}
\begin{enumerate}
    \item Compute the reduction coloring on the original data matrix \(X\), producing a partition \((\mathcal{P}, \mathcal{Q})\). Using this coloring, form the reduced matrix $X'$.
    \vspace{-2mm}
    \item Apply StandardScaler to \(X'\). Let \(\mu' \in \mathbb{R}^{|\mathcal{Q}|}\) and \(\sigma' \in \mathbb{R}^{|\mathcal{Q}|}\) denote the feature means and standard deviations of \(X'\). Train the desired model on 
    \vspace{-2mm}
    \[x'_{\mathrm{sc}} = D_{\sigma'}^{-1}(x' - \mu').
    \vspace{-3mm}
    \] For example, in logistic regression, the learned parameters satisfy 
    \vspace{-2mm}
    \[z = (w'_{\mathrm{sc}})^\top x'_{\mathrm{sc}} + b'_{\mathrm{sc}}.
    \vspace{-4mm}
    \]
    \vspace{-2mm}
    \item Convert the fitted parameters back to the (unscaled) reduced feature space. Substituting 
    \vspace{-2mm}
    \[x'_{\mathrm{sc}} = D_{\sigma'}^{-1}(x' - \mu'),
    \vspace{-3mm}
    \] we obtain
    \vspace{-2mm}
    \begin{align*}
        z &=(w'_{\mathrm{sc}})^\top D_{\sigma'}^{-1}(x' - \mu') + b'_{\mathrm{sc}} \\
        &= (D_{\sigma'}^{-1} w'_{\mathrm{sc}})^\top x' + \left(b_{\mathrm{sc}}' - (D_{\sigma'}^{-1} w'_{\mathrm{sc}})^\top \mu' \right).
        \vspace{-2mm}
    \end{align*}
    Hence, the parameters in the unstandardized reduced space are 
    \vspace{-2mm}
    \[w' = D_{\sigma'}^{-1} w'_{\mathrm{sc}}, \qquad b' = b'_{\mathrm{sc}} - (w')^\top \mu'.
    \vspace{-3mm}
    \]
    \item Lift \(w'\) to the original feature space using the column partition \(\mathcal{Q}\). The intercept remains \(b = b'\).
    \vspace{-4mm}
\end{enumerate}

At inference time, predictions may be computed directly on the original (unstandardized) feature vectors \(x\) using \(
z = w^\top x + b.\) No additional centering or normalization is required, since the effect of standardization has already been absorbed into the transformed parameters \(w\) and \(b\).
This procedure ensures that training on the standardized reduced data is exactly equivalent (up to solver tolerance) to training on a standardized version of the original data under the reduction induced by \((\mathcal{P}, \mathcal{Q})\).

\subsection{Kernel-Based Models}
\vspace{-1mm}

Finally, we consider kernel methods, where learning operates only on pairwise
inner products $k(x_i,x_j)$ rather than explicit features. 

After applying the standard representer theorem-based reduction, training reduces to a convex program in coefficients $\alpha\in\mathbb{R}^n$ whose structure is determined by the symmetric positive semidefinite kernel matrix $K_{ij}=k(x_i,x_j)$ with $K\in\mathbb{R}^{n\times n}$. Theorem~\ref{reductiontheorem} applies directly to this finite problem. Here, we color $\alpha$ by $\mathcal{Q}$.

\subsubsection{Kernel Ridge Regression}\label{sec:kernelridge}
Training of weighted kernel ridge regression reduces to
\vspace{-2mm}
\begin{align*}
    \minimize \ F(\alpha, b) =
    &(K\alpha + b1_n - y)^\top W(K\alpha + b1_n - y)\\
    &\qquad+ \lambda\alpha^\top K \alpha,
\end{align*}
\vspace{-10mm}

where $y\in\mathbb{R}^n$ is the ground truth and $W=\text{diag}(v)$ for sample weights $v\in\mathbb{R}^n$.

\begin{theorem}\label{kernellinregreduction}
A coloring $\mathcal{Q}$ of $\alpha$ is a reduction coloring for kernel ridge regression under the following conditions:
\vspace{-3mm}
\begin{itemize}
    \item $(\mathcal{Q}, \mathcal{Q})$ is equitable on $K$.
    \vspace{-2mm}
    \item For each $T\in\mathcal{Q}$, $\sum_{i=1}^n v_iK_{ij}y_i$ is constant across $j\in T$.
    \vspace{-2mm}
    \item $v_i$ is constant across all $i\in T$ for any $T\in\mathcal{Q}$.
    \vspace{-3mm}
\end{itemize}
\end{theorem}
Note that this is nearly identical to the conditions for primal linear regression under Theorem~\ref{linregreduction} and, by extension, primal ridge regression. We prove Theorem~\ref{kernellinregreduction} in Appendix~\ref{kernellinregreductionproof}.
We can express the reduced problem using $K'=\Pi_\mathcal{Q}^\top K\Pi_\mathcal{Q}$, replacing $1_n$ with $1_{|\mathcal{Q}|}$, $y'=\Pi_\mathcal{Q}^{\text{Scaled}}y$ (which, by similar logic as in Section~\ref{sec:linreg}, preserves the objective function up to shifting by a constant), and $W'=\Pi_\mathcal{Q}^\top W\Pi_\mathcal{Q}$ since $v'=\Pi_\mathcal{Q}^\top v$.

\subsubsection{Kernel Binary Logistic Regression}
\vspace{-1mm}

With binary labels $y\in\{0,1\}^n$, the optimization problem for kernel logistic regression is
\vspace{-2mm}
\begin{align*}
    \minimize &\ F(\alpha, b)=\\
    &-\left(\sum_{i=1}^n v_i\left(y_i \log \hat y_i + (1-y_i)\log(1-\hat y_i)\right)\right)\\
    &\qquad+ \lambda\alpha^\top K \alpha,
\end{align*}
\vspace{-9mm}

where $\hat y = \sigma(K\alpha + b1_n)$ and $\sigma$ is the sigmoid function.

\begin{theorem}\label{kernelblrreduction}
A coloring $\mathcal{Q}$ is a reduction coloring for kernel binary logistic regression under the exact same conditions as for kernel ridge regression from Theorem~\ref{kernellinregreduction}.
\end{theorem}
This will be proved in Appendix~\ref{kernelblrreductionproof}. The reduced problem can be expressed in the same way as kernel ridge regression in Section~\ref{sec:kernelridge}.

To avoid fractional values in the ground truth, we can alternatively merge samples in each $T\in\mathcal{Q}$ into up to two reduced weighted samples, one with each label. In particular, each $T\in\mathcal{Q}$ is split into up to two subsets $T^{(0)}=\{i\in T\mid y_i=0\}$ and $T^{(1)}=\{i\in T\mid y_i=1\}$. The resulting partition is denoted by $\mathcal{P}$. In the reduced problem, the quadratic regularization term depends only on the projected coefficients and therefore uses the symmetrically projected matrix $K'_\text{quad}=\Pi_\mathcal{Q}^\top K\Pi_\mathcal{Q}$. The prediction term, however, depends on the logits $z=K\alpha+b1_n$. After aggregating identical samples with identical labels by $\mathcal{P}$, the logits are determined by the row-aggregated matrix 
\vspace{-2mm}
\[K'_\text{logit}=\Pi_\mathcal{P}^{\text{Scaled}}K\Pi_\mathcal{Q}.
\vspace{-2mm}
\] Thus, the reduced objective uses two derived matrices: one symmetric matrix for the quadratic regularization term and one matrix with a potentially different row aggregation procedure for the loss. This asymmetry reflects the fact that the regularizer depends only on $\alpha$, whereas the loss depends on logits indexed by training examples. Hence, the reduced objective can be written as follows:
\vspace{-2mm}
\begin{align}
    \minimize &\ F'(\alpha', b')=\nonumber\\
    &-\left(\sum_{S\in\mathcal{P}} v'_S\left(y'_S \log \hat{y}'_S + (1-y'_S)\log(1-\hat{y}'_S)\right)\right)\nonumber\\
    &\qquad+ \lambda\alpha'^\top K'_\text{quad} \alpha',\nonumber
\end{align}
\vspace{-8mm}

where $v'=\Pi_\mathcal{P}^\top v$, $y'=\Pi_\mathcal{P}^{\text{Scaled}}y$, and $\hat{y}'=\sigma(K'_\text{logit}\alpha'+b1_{|\mathcal{P}|})$.

We note that while the reduced problem is a smaller convex program equivalent to the original, it is not necessarily expressible again in standard kernel form.

\section{Experiments}\label{sec:experiments}

We empirically evaluate our exact reduction framework on binary logistic
regression.
The experiments are designed to answer three questions:
(i) how much instance and feature compression can be achieved in practice,
(ii) whether the reduction overhead is outweighed by solver speedups, and
(iii) how these effects vary across datasets with different scales and sparsity
patterns.
As predicted by theory, all reduced problems yield solutions that are equivalent
to those obtained from the original datasets.

\subsection{Datasets}

We evaluate our method on one dataset from OpenML and four standard benchmark
datasets from LIBSVM, covering a range of sample sizes, feature dimensions,
and data characteristics.

We use the Titanic dataset accessed via the OpenML Python API \citep{openml}.
The following preprocessing steps are applied:
\vspace{-3mm}
\begin{itemize}
    \item Categorical features are transformed using one-hot encoding, with the
    first category dropped to avoid redundancy.
    \vspace{-2mm}
    \item Missing categorical values are treated as separate categories.
    \vspace{-2mm}
    \item Missing numerical values are imputed using mean imputation.
    \vspace{-2mm}
    \item Identifying or high-cardinality features (name, ticket, cabin, boat,
    and body) are removed.
    \vspace{-3mm}
\end{itemize}

We also evaluate on the following binary classification datasets from
LIBSVM \citep{libsvm}:
skin\_nonskin, phishing, a7a, and
breast-cancer.
These datasets are commonly used to benchmark large-scale convex optimization
methods and provide a diverse testbed for evaluating reduction behavior.

\subsection{Experimental Setup}

All reduction procedures are implemented in C++, while model training
is performed using scikit-learn \citep{scikit-learn}.
We apply StandardScaler prior to training to improve numerical stability;
the exact parameter transformation ensuring equivalence between reduced and
original problems is described in Section~\ref{sec:standardscaler}.

For each dataset, we measure:
\vspace{-2mm}
\begin{itemize}
    \vspace{-2mm}
    \item the number of samples and features retained after reduction,
    \vspace{-2mm}
    \item and the total runtime, defined as the sum of reduction time and training
    time on the reduced dataset.
    \vspace{-3mm}
\end{itemize}
To facilitate comparison across datasets, all runtimes are normalized by the
training time of logistic regression on the original (unreduced) dataset.
Predictive performance is identical up to solver tolerance and is therefore not reported separately.

\subsection{Results}

We first examine the extent of sample and feature compression achieved by exact reduction in Figure~\ref{fig:sample-feature-reduction}.
\vspace{-3mm}
\begin{figure}[H]
  \centering
  \includegraphics[width=\columnwidth]{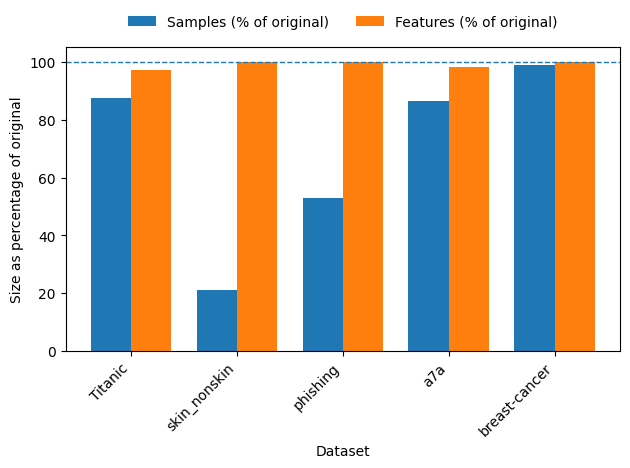}
  \caption{Sample and feature compression percentage by dataset.}
  \label{fig:sample-feature-reduction}
  \vspace{-5mm}
\end{figure}
Across all datasets, we observe compression in number of features or samples, or reduction in both features and samples.
These results highlight the prevalence of redundancies in real-world datasets that are directly exploitable by our framework.

We next consider computational efficiency.
Figure~\ref{fig:runtime} reports the total runtime with reduction, expressed as a
percentage of the baseline training runtime on the original dataset.

\vspace{-5mm}
\begin{figure}[H]
  \vskip 0.2in
  \begin{center}
    \centerline{\includegraphics[width=\columnwidth]{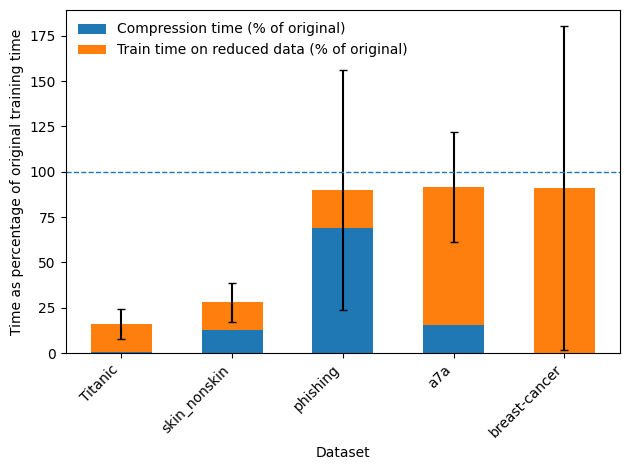}}
    \caption{
      Runtime compression percentage by dataset. Error bars denote the standard deviation over 50 independent trials.
    }
    \label{fig:runtime}
  \end{center}
  \vspace{-10mm}
\end{figure}

In most cases, the overhead incurred by color refinement is more than offset by
the reduced cost of solving a smaller optimization problem.
As a result, exact reduction leads to meaningful end-to-end speedups, even when
accounting for preprocessing time. Overall, these experiments demonstrate that our reduction framework provides tangible computational benefits
in practical machine learning pipelines, especially for large-scale convex ERM.

\section{Conclusion}\label{sec:conclusion}
We proposed an exact reduction framework for convex empirical risk minimization
based on equitable partitions and color refinement, showing that a broad class of
models, including linear and logistic regression, elastic-net regularization, and
kernel methods, admit nontrivial instance and feature aggregation without loss of
optimality.
The framework provides explicit, verifiable reduction conditions and a procedure for compression.
Experiments demonstrate that such exact reductions are not merely theoretical but
can reduce problem size and improve end-to-end training time on
real-world datasets.

The primary limitation of our work is that the effectiveness of reduction depends heavily on the extent of redundancies within a given dataset. Additionally, explicit compression algorithms are only feasible for optimization problems with known structure. Future work includes improving the efficiency of color refinement for equitability and applying our reduction theorem to more classes of convex empirical risk minimization.




\section{Acknowledgments}
We would like to thank the MIT Program for Research in Mathematics, Engineering, and Science (PRIMES-USA) for making this project possible. The work of Z. Chen is supported in part by the National Science Foundation via grant DMS-2509011.

\bibliography{references}
\bibliographystyle{icml2026}

\newpage

\appendix
\onecolumn
\section{Proof of Reduction Theorem}\label{app:reductiontheoremproof}
In this proof, we will occasionally use color-based indexing in cases where all elements of a color are known to be equal. For example, we might say $x_T$ for $T\in\mathcal{Q}$ if it is known that $x_j$ is equal for all $j\in T$.

\begin{proof}[Proof of Theorem~\ref{reductiontheorem}]
Suppose that $(\mathcal{P}, \mathcal{Q})$ is a reduction coloring. Let $\hat{x}=\Pi_\mathcal{Q}\Pi_\mathcal{Q}^{\text{Scaled}} x$. Note that if $j_1\in T$ and $T\in\mathcal{Q}$, then \[\hat{x}_{j_1}=\frac{\sum_{j_2\in T}x_{j_2}}{|T|}.\] We prove our reduction theorem in three steps:
\begin{enumerate}
    \item The objective function does not increase when $x$ is changed to $\hat{x}$.
    \item If $x$ is feasible, $\hat{x}$ must be too.
    \item The above statements imply equivalence of the original and reduced problems.
\end{enumerate}

Firstly, we show that $F(\hat{x})\leq F(x)$ using the Lagrangian, inspired by \cite{chen}. Suppose that $x$ is fixed, and we wish to minimize $F(z)$ such that \[\sum_{j\in T}z_j=\sum_{j\in T}x_j\] for all $T\in\mathcal{Q}$. It suffices to show that $z=\hat{x}$ is a minimum, and hence, since the objective is convex and it is obvious that $z=\hat{x}$ satisfies our condition, that the partial derivative of the Lagrangian \[L=F(z)-\sum_{T\in\mathcal{Q}}\lambda_T\left(\left(\sum_{j\in T}z_j\right)-\sum_{j\in T}x_j\right)\] with respect to each $z_{j_s}$ is 0 if $z=\hat{x}$. We have that \[\frac{\partial L}{\partial z_{j}}=\frac{\partial F}{\partial z_j}-\lambda_T\]
if $j\in T$. Then our claim is true if we have equal $\left.\frac{\partial F}{\partial z_j}\right|_{z=\hat{x}}$ for all $j$ in the same color $T$, which is one of our reduction coloring conditions.

Secondly, we show that $\hat{x}$ is feasible. Due to our requirement that each $T\in\mathcal{Q}$ has the same $(l_j, u_j)$ for all $j\in T$, we immediately have that $\hat{x}$ is in the box domain. Additionally, we know that \[\sum_{i\in S}G_i(\hat{x})\leq\sum_{i\in S}G_i(x)\leq |S|b_S.\] The first inequality follows by applying the same equality-constrained convex minimization argument to the convex function $\sum_{i\in S}G_i$ using the gradient condition from the reduction coloring for $\sum_{i\in S}G_i$. The second inequality is true by assumption and due to our requirement that $b_i$ is equal for all $i\in S$. In fact, $G_i(\hat{x})\leq b_i$ for all individual $1\leq i\leq m$ because we require that $G_i(\hat{x})$ is equal for all $i\in S$.

Finally, we show that the above implies optimality of $x'$ on the reduced problem. Since $F'$ is given simply by replacing each $x_j$ in $F$ with $x_T$ where $j\in T$, we have that \[F'(x')=F(\hat{x}),\]
so if $x'$ is not optimal, then $\hat{x}$ and hence $x$ is not optimal, which contradicts our assumption. We can use similar logic with the reduced constraints: $G'(x')=G(\hat{x})$, so $x'$ is feasible too.

For the reverse direction, note that $\hat{x}=\Pi_\mathcal{Q}x'$. Hence, if $x'$ is optimal, it is impossible for $\hat{x}$ to not be optimal. Otherwise, there exists some feasible $y$ with $F(y)<F(\hat{x})\implies F'(y')<F'(x')$, contradiction.
\end{proof}

\begin{remark}
    To see why the reduced objective and constraints are convex, note that every reduced expression is given by a scaled intersection of the original curve and the affine space given by setting all $x_j$ equal for all $j\in T$ and any fixed $T\in\mathcal{Q}$. If a function is convex over a convex set, then it is also convex over any affine subspace of this set \cite{loewen}, showing that the reduced expressions are convex.
\end{remark}

\section{Proof and Example of Reduction Coloring Coarseness Compared to Permutation Invariance}\label{app:coarsenesstheoremproof}

\begin{proof}[Proof of Theorem~\ref{coarsenesstheorem}]
Consider any two variables $j_1,j_2$ in the same orbit under $\Gamma$. Then there exists some $(\pi,\sigma)\in\Gamma$ such that $\pi(j_1)=j_2$. Let $\hat{x}$ be any point in the box domain that is constant within each $T\in\mathcal{Q}_\Gamma$, meaning $\pi(\hat{x})=\hat{x}$.

Since $F(\pi(x))=F(x)$ for all $x$ in the box domain, viewing $\pi$ as the permutation matrix $P_\pi$ (so $\pi(x)=P_\pi x$) and differentiating gives, by the chain rule, \[P_\pi^\top \nabla F(P_\pi x)=\nabla F(x)\] for all such $x$. Evaluating at $x=\hat{x}$ and using $P_\pi\hat{x}=\hat{x}$ yields \[P_\pi^\top \nabla F(\hat{x})=\nabla F(\hat{x}),\] which is equivalent to $\nabla F(\hat{x})=P_\pi \nabla F(\hat{x})$ since $P_\pi^{-1}=P_\pi^\top$ for any $\pi$. Therefore, $\big(\nabla F(\hat{x})\big)_{j}=\big(\nabla F(\hat{x})\big)_{\pi(j)}$ for all $j$, so \[\left.\frac{\partial F}{\partial x_{j_1}}\right|_{x=\hat{x}}=\left.\frac{\partial F}{\partial x_{j_2}}\right|_{x=\hat{x}}.\]

For constraints, fix any constraint orbit $S\in\mathcal{P}_\Gamma$. We claim that $\sum_{i\in S}G_i(\pi(x))=\sum_{i\in S}G_i(x)$ for all $x$ in the box domain. By the definition of $\Gamma$, we have $G_{\sigma(i)}(\pi(x))=G_i(x)$ for some permutation $\sigma$ for all $1\leq i\leq m$, so summing over $i\in S$ gives \[\sum_{i\in S} G_{\sigma(i)}(\pi(x))=\sum_{i\in S}G_i(x).\] Since $S$ is an orbit, it is $\Gamma$-invariant. Hence, $\sigma(S)=S$, and $\sigma$ restricts to a bijection on $S$. Thus, $\sum_{i\in S}G_{\sigma(i)}(\pi(x))$ is simply a reindexing: \[\sum_{i\in S} G_{\sigma(i)}(\pi(x))=\sum_{i\in S}G_i(\pi(x))=\sum_{i\in S}G_i(x).\] Using the same logic as above, this implies that \[\left.\frac{\partial}{\partial x_{j_1}}\left(\sum_{i\in S}G_i\right)\right|_{x=\hat{x}}=\left.\frac{\partial}{\partial x_{j_2}}\left(\sum_{i\in S}G_i\right)\right|_{x=\hat{x}}.\] Finally, by definition of $\Gamma$, constraints in the same orbit satisfy $b_{\sigma(i)}=b_i$ and $G_{\sigma(i)}(\pi(x))=G_i(x)$. Evaluating at $x=\hat{x}$ with $\pi(\hat{x})=\hat{x}$ implies that $G_i(\hat{x})$ and $b_i$ are both constant within each $S\in\mathcal{P}_\Gamma$. Moreover, $\pi(l)=l$ and $\pi(u)=u$ imply that $(l_j,u_j)$ is constant within each $T\in\mathcal{Q}_\Gamma$. Therefore, $(\mathcal{P}_\Gamma,\mathcal{Q}_\Gamma)$ is a reduction coloring, so the coarsest reduction coloring $(\mathcal{P},\mathcal{Q})$ must be at least as coarse as $(\mathcal{P}_\Gamma,\mathcal{Q}_\Gamma)$.
\end{proof}

Now, we provide an example where Theorem~\ref{reductiontheorem} can be used for compression even when no nontrivial permutation symmetry exists. Consider an unregularized linear regression problem on the following data: \[X=\begin{bmatrix} 3 & 1 & 0 \\ 1 & 3 & 0 \\ 4 & 2 & 2 \\ 2 & 4 & 2 \\ 3 & 3 & 1 \end{bmatrix}, \qquad y=\begin{bmatrix} 0 \\ 1 \\ 1 \\ 0 \\ 7 \end{bmatrix}.\] There is clearly no pair of permutations $(\pi, \sigma)$ that preserves the data. For example, swapping the first and second features sends the sample $(3, 1, 0)$ with label 0 to $(1, 3, 0)$, which occurs only with label 1, so no sample permutation can restore equality. However, $\mathcal{P}=\{\{1, 2\}, \{3, 4\}, \{5\}\}$ and $\mathcal{Q}=\{\{1, 2\}, \{3\}\}$ is a reduction coloring. Thus, the problem is equivalent to one with the following compressed data: \[X'=\begin{bmatrix} 4 & 0 \\ 6 & 2 \\ 6 & 1 \end{bmatrix}, \qquad y'=\begin{bmatrix} 0.5 \\ 0.5 \\ 7 \end{bmatrix},\] with \[W'=\begin{bmatrix} 2 & 0 & 0 \\ 0 & 2 & 0 \\ 0 & 0 & 1 \end{bmatrix}.\] We have that $w=\begin{bmatrix} 6.5 \\ 6.5 \\ -6.5 \end{bmatrix}$, $b=b'=-25.5$, and $w'=\begin{bmatrix} 6.5 \\ -6.5 \end{bmatrix}$, as expected.

\section{Proofs and Examples for Section~\ref{sec:ml}}
\subsection{Proof of Reduction for Linear Regression and Preservation of Closed-Form}\label{app:linregproof}
\begin{proof}[Proof of Theorem~\ref{linregreduction}]
At some point $\hat{w}$ with $\hat{w}_{j_1}=\hat{w}_{j_2}$ if $j_1$ and $j_2$ share a color in $\mathcal{Q}$,
\begin{align*}
    \left.\frac{\partial F}{\partial w_j}\right|_{w=\hat{w}} &= 2\sum_{i=1}^nX_{ij}(\hat{y}_i - y_i)\\
    &=2\sum_{i=1}^n X_{ij}\left(\left(b+\sum_{k=1}^D X_{ik}w_k\right) - y_i\right)\\
    &=2\sum_{i=1}^n X_{ij}\left(\left(b+\sum_{T\in\mathcal{Q}}w_T\sum_{k\in T} X_{ik}\right) - y_i\right),
\end{align*}
where we use $w_T$ to denote the constant value of $w_j$ for all $j\in T$.

Under our reduction conditions, for any $T\in\mathcal{Q}$ and $j_1,j_2\in T$,
\begin{align*}
    \left(\left.\frac{\partial F}{\partial w_{j_1}}\right|_{w=\hat{w}}\right)-\left(\left.\frac{\partial F}{\partial w_{j_2}}\right|_{w=\hat{w}}\right)&=2\sum_{i=1}^n (X_{ij_1}-X_{ij_2})\left(\left(b+\sum_{T\in\mathcal{Q}}w_T\sum_{k\in T} X_{ik}\right) - y_i\right)\\
    &=2\sum_{i=1}^n (X_{ij_1}-X_{ij_2})\left(b+\sum_{T\in\mathcal{Q}}w_T\sum_{k\in T}X_{ik}\right)-\left(\sum_{i=1}^n X_{ij_1}y_i-X_{ij_2}y_i\right) \\
    &=2\sum_{i=1}^n (X_{ij_1}-X_{ij_2})\left(b+\sum_{T\in\mathcal{Q}}w_T\sum_{k\in T}X_{ik}\right).
\end{align*}
Due to equitability of $(\mathcal{P}, \mathcal{Q})$ on $X$, we have constant $\sum_{j\in T}X_{ij}$ for all $i\in S$ and constant $\sum_{i\in S}X_{ij}$ for all $j\in T$ if $S\in\mathcal{P}$ and $T\in\mathcal{Q}$. Hence, the prediction for each sample, \[\hat{y}_i=b+\sum_{T\in\mathcal{Q}}w_T\sum_{j\in T}X_{ij},\] is equal for all $i$ that share a color $S\in\mathcal{P}$, and we can denote this value as $\hat{y}_S$. Using this, we can further simplify:
\begin{align*}
    \sum_{i=1}^n (X_{ij_1}-X_{ij_2})\left(b+\sum_{T\in\mathcal{Q}}w_T\sum_{k\in T}X_{ik}\right)&= \sum_{S\in\mathcal{P}}\hat{y}_S\left(\left(\sum_{i\in S}X_{ij_1}\right)-\sum_{i\in S}X_{ij_2}\right)\\
    &= \sum_{S\in\mathcal{P}}\hat{y}_S\cdot 0\\
    &=0.
\end{align*}
Hence, $\left.\frac{\partial F}{\partial w_{j_1}}\right|_{w=\hat{w}}=\left.\frac{\partial F}{\partial w_{j_2}}\right|_{w=\hat{w}}$ for all $j_1,j_2\in T$.
\end{proof}

We can further verify that the closed-form optimal $w$, \[w=\left(X^\top WX\right)^{-1}X^\top Wy,\] is preserved. (Note that we assume $W=I$ in the original problem. Additionally, we assume $X$ has been augmented with a column of ones so that $w$ includes a bias term.) In the reduced problem, this formula becomes 
\begin{align*}
    w'&=\left((X')^\top W'X'\right)^{-1}(X')^\top W'\Pi^{\text{Scaled}}_\mathcal{P}y\\
    &=\left(\Pi_\mathcal{Q}^\top X^\top \left(\Pi_\mathcal{P}^{\text{Scaled}}\right)^\top W'\Pi_\mathcal{P}^{\text{Scaled}}X\Pi_\mathcal{Q}\right)^{-1}\Pi_\mathcal{Q}^\top X^\top \left(\Pi_\mathcal{P}^{\text{Scaled}}\right)^\top W'\Pi^{\text{Scaled}}_\mathcal{P}y,
\end{align*}
which implies that \[\Pi_\mathcal{Q}^\top X^\top \left(\Pi_\mathcal{P}^{\text{Scaled}}\right)^\top \Pi_\mathcal{P}^\top\Pi_\mathcal{P}\Pi_\mathcal{P}^{\text{Scaled}}X\Pi_\mathcal{Q}w'=\Pi_\mathcal{Q}^\top X^\top \left(\Pi_\mathcal{P}^{\text{Scaled}}\right)^\top \Pi_\mathcal{P}^\top\Pi_\mathcal{P}\Pi^{\text{Scaled}}_\mathcal{P}y.\]
Since $\Pi_\mathcal{P}\Pi_\mathcal{P}^{\text{Scaled}}\in\mathbb{R}^{n\times n}$ satisfies $\left(\Pi_\mathcal{P}\Pi_\mathcal{P}^{\text{Scaled}}\right)_{i_1i_2}=\frac{1}{|S|}$ for any $i_1, i_2\in S$, left-multiplication by $\left(\Pi_\mathcal{P}^{\text{Scaled}}\right)^\top\Pi_\mathcal{P}^\top\Pi_\mathcal{P}\Pi_\mathcal{P}^{\text{Scaled}}$ averages the entries within each color $S\in\mathcal{P}$. Additionally, if $i_1, i_2\in S$, then we have that
\begin{align*}
    (X\hat{w})_{i_1}&=\sum_{T\in\mathcal{Q}}w'_T\sum_{j\in T}X_{i_1j}\\
    &=\sum_{T\in\mathcal{Q}}w'_T\sum_{j\in T}X_{i_2j}\\
    &=(X\hat{w})_{i_2}
\end{align*}
from equitability. Thus,
\begin{equation}\label{eq:useful1}
    \left(\Pi_\mathcal{P}^{\text{Scaled}}\right)^\top\Pi_\mathcal{P}^\top\Pi_\mathcal{P}\Pi_\mathcal{P}^{\text{Scaled}}X\hat{w}=X\hat{w},
\end{equation}
so we can simplify: \[\Pi_\mathcal{Q}^\top X^\top X\hat{w}=\Pi_\mathcal{Q}^\top X^\top \left(\Pi_\mathcal{P}^{\text{Scaled}}\right)^\top \Pi_\mathcal{P}^\top\Pi_\mathcal{P}\Pi^{\text{Scaled}}_\mathcal{P}y.\] In other words, for any $T\in\mathcal{Q}$, we have that \[\sum_{j\in T}\left(X^\top X\hat{w}\right)_j=\sum_{j\in T}\left(X^\top \left(\Pi_\mathcal{P}^{\text{Scaled}}\right)^\top \Pi_\mathcal{P}^\top\Pi_\mathcal{P}\Pi^{\text{Scaled}}_\mathcal{P}y\right)_j.\] We can rewrite $\left(X^\top X\hat{w}\right)_j$ as follows:
\begin{align*}
    \left(X^\top X\hat{w}\right)_j&=\sum_{k=1}^D\left(X^\top X\right)_{jk}\hat{w}_k\\
    &=\sum_{k=1}^D\hat{w}_kX_k^\top X_j\\
    &=\left(\sum_{T\in\mathcal{Q}}w'_T\sum_{k\in T}X_k\right)^\top X_j\\
    &=\sum_{i=1}^n\left(\sum_{T\in\mathcal{Q}}w'_T\sum_{k\in T}X_{ik}\right)X_{ij},
\end{align*}
where $X_j$ is the column vector representing the $j$th feature. We have that
\begin{align*}
    \sum_{i=1}^n\left(\sum_{T\in\mathcal{Q}}w'_T\sum_{k\in T}X_{ik}\right)X_{ij}&=\sum_{S\in\mathcal{P}}\sum_{T\in\mathcal{Q}}w'_T\sum_{i\in S}X_{ij}\sum_{k\in T}X_{ik}\\
    &=\sum_{S\in\mathcal{P}}\sum_{T\in\mathcal{Q}}w'_TX'_{ST}\sum_{i\in S}X_{ij}
\end{align*}
from equitability of $(\mathcal{P}, \mathcal{Q})$ on $X$. It also follows from equitability that $\sum_{i\in S}X_{ij_1}=\sum_{i\in S}X_{ij_2}$ if $j_1$ and $j_2$ share a color in $\mathcal{Q}$, so \[\sum_{S\in\mathcal{P}}\sum_{T\in\mathcal{Q}}w'_TX'_{ST}\sum_{i\in S}X_{ij}=\left(X^\top X\hat{w}\right)_j\] is constant across all $j\in T$. Recall that \[\sum_{j\in T}\left(X^\top X\hat{w}\right)_j=\sum_{j\in T}\left(X^\top \left(\Pi_\mathcal{P}^{\text{Scaled}}\right)^\top \Pi_\mathcal{P}^\top\Pi_\mathcal{P}\Pi^{\text{Scaled}}_\mathcal{P}y\right)_j\] for any $T\in\mathcal{Q}$. Since left-multiplication by $\left(\Pi_\mathcal{P}^{\text{Scaled}}\right)^\top\Pi_\mathcal{P}^\top\Pi_\mathcal{P}\Pi_\mathcal{P}^{\text{Scaled}}$ averages each color $S\in\mathcal{P}$, $\left(\left(\Pi_\mathcal{P}^{\text{Scaled}}\right)^\top \Pi_\mathcal{P}^\top\Pi_\mathcal{P}\Pi^{\text{Scaled}}_\mathcal{P}y\right)_i$ is equal for all $i\in S$, and we can express it as $\left(\left(\Pi_\mathcal{P}^{\text{Scaled}}\right)^\top \Pi_\mathcal{P}^\top\Pi_\mathcal{P}\Pi^{\text{Scaled}}_\mathcal{P}y\right)_S$. This means that \[\left(X^\top \left(\Pi_\mathcal{P}^{\text{Scaled}}\right)^\top \Pi_\mathcal{P}^\top\Pi_\mathcal{P}\Pi^{\text{Scaled}}_\mathcal{P}y\right)_j=\sum_{S\in\mathcal{P}}\left(\left(\Pi_\mathcal{P}^{\text{Scaled}}\right)^\top \Pi_\mathcal{P}^\top\Pi_\mathcal{P}\Pi^{\text{Scaled}}_\mathcal{P}y\right)_S\sum_{i\in S} X_{ij}\] is also constant across $j\in T$ due to equitability. These show that
\begin{equation}\label{eq:linreg}
    X^\top X\hat{w}=X^\top \left(\Pi_\mathcal{P}^{\text{Scaled}}\right)^\top \Pi_\mathcal{P}^\top\Pi_\mathcal{P}\Pi^{\text{Scaled}}_\mathcal{P}y.
\end{equation}
We also have for any $T\in\mathcal{Q}$ that
\begin{align*}
    \sum_{j\in T}\left(X^\top \left(\Pi_\mathcal{P}^{\text{Scaled}}\right)^\top \Pi_\mathcal{P}^\top\Pi_\mathcal{P}\Pi^{\text{Scaled}}_\mathcal{P}y\right)_j&=\sum_{j\in T}\sum_{S\in\mathcal{P}}y'_S\sum_{i\in S}X_{ij}\\
    &=\sum_{S\in\mathcal{P}}y'_S\sum_{i\in S}\sum_{j\in T}X_{ij}\\
    &=\sum_{S\in\mathcal{P}}|S|\,y'_S\,X_{ST}\\
    &=\sum_{S\in\mathcal{P}}\sum_{i\in S}y_i X_{ST}\\
    &=\sum_{S\in\mathcal{P}}\sum_{i\in S}y_i\sum_{j\in T}X_{ij}\\
    &=\sum_{i=1}^n y_i\sum_{j\in T}X_{ij}\\
    &=\sum_{j\in T}\left(X^\top y\right)_j,
\end{align*}
where $X_{ST}=\sum_{j\in T}X_{ij}$ is constant regardless of our choice of $i\in S$ due to equitability. Since we require that $\left(X^\top y\right)_j$ is equal across $j\in T$ and \[\left(X^\top \left(\Pi_\mathcal{P}^{\text{Scaled}}\right)^\top \Pi_\mathcal{P}^\top\Pi_\mathcal{P}\Pi^{\text{Scaled}}_\mathcal{P}y\right)_j=\sum_{S\in\mathcal{P}}y'_S\sum_{i\in S}X_{ij}\] is also equal across $j\in T$ due to equitability, we have that
\begin{equation}\label{eq:useful2}
    X^\top \left(\Pi_\mathcal{P}^{\text{Scaled}}\right)^\top \Pi_\mathcal{P}^\top\Pi_\mathcal{P}\Pi^{\text{Scaled}}_\mathcal{P}y=X^\top y,
\end{equation}
so Equation~\ref{eq:linreg} becomes \[X^\top X\hat{w}=X^\top y\implies\hat{w}=\left(X^\top X\right)^{-1}X^\top y.\]

\subsection{Proof of Reduction for Binary Logistic Regression}\label{app:blrproof}
\begin{proof}[Proof of Theorem~\ref{blrreduction}]
If $\hat{w}_{j_1}=\hat{w}_{j_2}$ for any $j_1$ and $j_2$ that share a color in $\mathcal{Q}$,
\begin{align*}
    \left.\frac{\partial F}{\partial w_j}\right|_{w=\hat{w}} &= \sum_{i=1}^n v_iX_{ij}\left(\hat{y}_i - y_i\right)\\
    &= \sum_{i=1}^n v_iX_{ij}\left(\sigma\!\left(b+\sum_{k=1}^D X_{ik}w_k\right) - y_i\right)\\
    &=\sum_{i=1}^n v_iX_{ij}\left(\sigma\!\left(b+\sum_{T\in\mathcal{Q}}w_T\sum_{k\in T} X_{ik}\right) - y_i\right).
\end{align*}

Under our reduction conditions, for any $T\in\mathcal{Q}$ and $j_1,j_2\in T$,
\begin{align*}
    \left(\left.\frac{\partial F}{\partial w_{j_1}}\right|_{w=\hat{w}}\right)-\left(\left.\frac{\partial F}{\partial w_{j_2}}\right|_{w=\hat{w}}\right)&=\sum_{i=1}^n v_i(X_{ij_1}-X_{ij_2})\left(\sigma\!\left(b+\sum_{T\in\mathcal{Q}}w_T\sum_{k\in T} X_{ik}\right) - y_i\right)\\
    &=\left(\sum_{i=1}^n v_i(X_{ij_1}-X_{ij_2})\,\sigma\!\left(b+\sum_{T\in\mathcal{Q}}w_T\sum_{j\in T}X_{ij}\right)\right)-\sum_{i=1}^n v_i\left(X_{ij_1}y_i-X_{ij_2}y_i\right) \\
    &=\sum_{i=1}^n v_i(X_{ij_1}-X_{ij_2})\,\sigma\!\left(b+\sum_{T\in\mathcal{Q}}w_T\sum_{j\in T}X_{ij}\right)\\
    &=\sum_{S\in\mathcal{P}} \frac{v'_S}{|S|}\,\sigma(z_S)\left(\left(\sum_{i\in S}X_{ij_1}\right)-\sum_{i\in S}X_{ij_2}\right)\\
    &= \sum_{S\in\mathcal{P}} \frac{v'_S}{|S|}\,\sigma(z_S)\cdot 0\\
    &=0,
\end{align*}
where $z_S$ represents the value of the logit \[z_i=b+\sum_{T\in\mathcal{Q}}w_T\sum_{j\in T}X_{ij},\] which is equal for all $i\in S$ for some $S\in\mathcal{P}$. Hence, $\left.\frac{\partial F}{\partial w_{j_1}}\right|_{w=\hat{w}}=\left.\frac{\partial F}{\partial w_{j_2}}\right|_{w=\hat{w}}$ for all $j_1,j_2\in T$.
\end{proof}

\subsection{Proof of Reduction for Multiclass Logistic Regression}\label{app:mclrproof}
\begin{proof}[Proof of Theorem~\ref{mclrreduction}]
We have that
\begin{align*}
    \frac{\partial F}{\partial W_{jc}}&=\sum_{i=1}^n v_iX_{ij}(\hat{y}_{ic}-1\{y_i=c\})\\
    &=\sum_{i=1}^n v_iX_{ij}\left(\frac{e^{b_c+\sum_{k=1}^DX_{ik}W_{kc}}}{\sum_{l=1}^K e^{b_l+\sum_{k=1}^DX_{ik}W_{kl}}}-1\{y_i=c\}\right).
\end{align*}

If $(\mathcal{P}, \mathcal{Q})$ is a reduction coloring and the two rows $\hat{W}_{j_1}=\hat{W}_{j_2}$ whenever $j_1$ and $j_2$ share a color in $\mathcal{Q}$ (the entries within $\hat{W}_{j_1}$ and $\hat{W}_{j_2}$ do not necessarily have to be equal), then
\begin{align*}
    \left.\frac{\partial F}{\partial W_{jc}}\right|_{W=\hat{W}}&=\sum_{i=1}^n v_iX_{ij}\left(\frac{e^{b_c+\sum_{T\in\mathcal{Q}}\sum_{k\in T}X_{ik}W_{kc}}}{\sum_{l=1}^K e^{b_l+\sum_{T\in\mathcal{Q}}\sum_{k\in T}X_{ik}W_{kl}}}-1\{y_i=c\}\right)\\
    &=\sum_{i=1}^n v_iX_{ij}\left(\frac{e^{b_c+\sum_{T\in\mathcal{Q}}W_{Tc}\sum_{k\in T}X_{ik}}}{\sum_{l=1}^K e^{b_l+\sum_{T\in\mathcal{Q}}W_{Tl}\sum_{k\in T}X_{ik}}}-1\{y_i=c\}\right)\\
    &=\sum_{S\in\mathcal{P}}\sum_{i\in S} v_iX_{ij}\left(\hat{y}_{Sc}-1\{y_i=c\}\right),
\end{align*}
where \[\hat{y}_{Sc}=\frac{e^{b_c+\sum_{T\in\mathcal{Q}}W_{Tc}\sum_{k\in T}X_{ik}}}{\sum_{l=1}^K e^{b_l+\sum_{T\in\mathcal{Q}}W_{Tl}\sum_{k\in T}X_{ik}}}\] is constant across $i\in S$ if $(\mathcal{P}, \mathcal{Q})$ is equitable on $X$.

Then under reduction conditions, for any $T\in\mathcal{Q}$ and $j_1,j_2\in T$,
\begin{align*}
    \left(\left.\frac{\partial F}{\partial W_{j_1c}}\right|_{W=\hat{W}}\right)-\left(\left.\frac{\partial F}{\partial W_{j_2c}}\right|_{W=\hat{W}}\right)&=\sum_{S\in\mathcal{P}}\sum_{i\in S} v_i(X_{ij_1}-X_{ij_2})\left(\hat{y}_{Sc}-1\{y_i=c\}\right)\\
    &=\left(\sum_{S\in\mathcal{P}}\frac{v'_S}{|S|}\hat{y}_{Sc}\sum_{i\in S}(X_{ij_1}-X_{ij_2})\right)-\sum_{i=1}^nv_i(X_{ij_1}-X_{ij_2})\cdot 1\{y_i=c\}\\
    &=\sum_{S\in\mathcal{P}}\frac{v'_S}{|S|}\hat{y}_{Sc}\left(\left(\sum_{i\in S}X_{ij_1}\right)-\sum_{i\in S}X_{ij_2}\right)\\
    &=0.
\end{align*}
Hence, $\left.\frac{\partial F}{\partial W_{j_1c}}\right|_{W=\hat{W}}=\left.\frac{\partial F}{\partial W_{j_2c}}\right|_{W=\hat{W}}$ for all $j_1,j_2\in T$.
\end{proof}

\subsection{Preservation of Ridge Regression Closed-Form}\label{app:enetproof}
We can verify that the closed-form solution for ridge regression, \[w=\left(X^\top WX+\lambda_2I\right)^{-1}X^\top Wy,\] is preserved by our reduction coloring for linear regression. Note that the regularization term in the reduced problem is \[\sum_{T\in\mathcal{Q}}\lambda_2|T|(w'_T)^2.\] Hence, in the reduced problem, we have \[\left(\Pi_\mathcal{Q}^\top X^\top \left(\Pi_\mathcal{P}^{\text{Scaled}}\right)^\top W'\Pi_\mathcal{P}^{\text{Scaled}}X\Pi_\mathcal{Q}+Q'\right)w'=\Pi_\mathcal{Q}^\top X^\top \left(\Pi_\mathcal{P}^{\text{Scaled}}\right)^\top W'\Pi_\mathcal{P}^{\text{Scaled}}y,\] where $Q'$ is a diagonal matrix given by $Q'_{TT}=\lambda_2|T|$ for all $T\in\mathcal{Q}$. Because of Equation~\ref{eq:useful1} and Equation~\ref{eq:useful2} in Appendix~\ref{app:linregproof}, we can rewrite this as \[\Pi_\mathcal{Q}^\top X^\top X\hat{w}+Q'w'=\Pi_\mathcal{Q}^\top X^\top y.\] This implies that \[\lambda_2|T|w'_T+\sum_{j\in T}\left(X^\top X\hat{w}\right)_j=\sum_{j\in T}\left(X^\top y\right)_j\] for any $T\in\mathcal{Q}$. But since $\left(X^\top X\hat{w}\right)_j$ and $\left(X^\top y\right)_j$ are both constant across $j\in T$, dividing by $|T|$ gives \[\lambda_2\hat{w}+X^\top X\hat{w}=X^\top y\implies\hat{w}=\left(X^\top X+\lambda_2I\right)^{-1}X^\top y.\]

\subsection{Proof of Reduction for Kernel Ridge Regression}\label{kernellinregreductionproof}
\begin{proof}[Proof of Theorem~\ref{kernellinregreduction}]
We have that \[\left.\frac{\partial F}{\partial \alpha_{j}}\right|_{\alpha=\hat{\alpha}}=2\left(\sum_{i=1}^n v_iK_{ij}(K\hat{\alpha}+b1_n-y)_i\right)+2\lambda (K\hat{\alpha})_j.\] Suppose that $T_1\in\mathcal{Q}$ and $j_1, j_2\in T_1$. From equitability on $K$,
\begin{align}
    (K\hat{\alpha})_{j_1}&=\sum_{i=1}^n K_{ij_1}\hat{\alpha}_i\nonumber\\
    &=\sum_{T_2\in\mathcal{Q}}\hat{\alpha}_{T_2}\sum_{i\in T_2}K_{ij_1}\nonumber\\
    &=\sum_{T_2\in\mathcal{Q}}\hat{\alpha}_{T_2}\sum_{i\in T_2}K_{ij_2}\nonumber\\
    &=(K\hat{\alpha})_{j_2},\nonumber
\end{align}
so the $2\lambda (K\hat{\alpha})_j$ term is constant across all $j$ in the same color. Also, since $v_i$ is constant on each $T\in\mathcal{Q}$, we can write $v_i = \frac{v'_T}{|T|}$ for all $i\in T$. Hence, we can decompose the summation term of the derivative with respect to $\alpha_{j_1}$ as
\begin{align}
    &\left(\sum_{i=1}^nv_iK_{ij_1}(K\hat{\alpha})_i\right)+\left(\sum_{i=1}^nv_iK_{ij_1}b\right)-\sum_{i=1}^nv_iK_{ij_1}y_i\nonumber\\
    &\qquad =\left(\sum_{T\in\mathcal{Q}}\frac{v'_T}{|T|}(K\hat{\alpha})_T\sum_{i\in T}K_{ij_1}\right)+b\left(\sum_{T\in\mathcal{Q}}\frac{v'_T}{|T|}\sum_{i\in T}K_{ij_1}\right)-\sum_{i=1}^nv_iK_{ij_1}y_i\nonumber\\
    &\qquad =\left(\sum_{T\in\mathcal{Q}}\frac{v'_T}{|T|}(K\hat{\alpha})_T\sum_{i\in T}K_{ij_2}\right)+b\left(\sum_{T\in\mathcal{Q}}\frac{v'_T}{|T|}\sum_{i\in T}K_{ij_2}\right)-\sum_{i=1}^nv_iK_{ij_2}y_i\nonumber\\
    &\qquad =\left(\sum_{i=1}^nv_iK_{ij_2}(K\hat{\alpha})_i\right)+\left(\sum_{i=1}^nv_iK_{ij_2}b\right)-\sum_{i=1}^nv_iK_{ij_2}y_i,\nonumber
\end{align}
which proves Theorem~\ref{kernellinregreduction}.
\end{proof}

\subsection{Proof of Reduction for Kernel Binary Logistic Regression}\label{kernelblrreductionproof}
\begin{proof}[Proof of Theorem~\ref{kernelblrreduction}]
The derivative of the loss function is
\begin{align}
    \left.\frac{\partial F}{\partial \alpha_{j}}\right|_{\alpha=\hat{\alpha}}&=\left(\sum_{i=1}^n v_iK_{ij}(\hat{y}_i-y_i)\right)+2\lambda (K\hat{\alpha})_j\nonumber\\
    &=\left(\sum_{i=1}^n v_iK_{ij}\,\sigma((K\hat{\alpha})_i+b)\right)-\left(\sum_{i=1}^n v_iK_{ij}y_i\right)+2\lambda (K\hat{\alpha})_j\nonumber\\
    &=\left(\sum_{T\in\mathcal{Q}}\frac{v'_T}{|T|}\sigma((K\hat{\alpha})_T+b)\sum_{i\in T}K_{ij}\right)-\left(\sum_{i=1}^n v_iK_{ij}y_i\right)+2\lambda (K\hat{\alpha})_j,\nonumber
\end{align}
where the last equality is true because we know from the proof of Theorem~\ref{kernellinregreduction} in Appendix~\ref{kernellinregreductionproof} that $(K\hat{\alpha})_j$ is constant across $j\in T$ for any $T\in\mathcal{Q}$. This derivative as a whole is constant across $j\in T$ because we have from equitability that the first term in the above expression is constant, and from Appendix~\ref{kernellinregreductionproof} that the latter two terms are also constant.
\end{proof}

\section{Full Experimental Results}\label{fullresults}
We used a binary logistic regression model with regularization turned off (no penalty term) and a maximum of 10,000 iterations. All remaining hyperparameters were set to their default values in scikit-learn \cite{scikit-learn}.

The following table shows the number of reduced samples and reduced features for each dataset. We used a binary ground truth without fractional values, so the number of reduced samples may be greater than $|\mathcal{P}|$.
\begin{table}[H]
  \caption{Size comparison between original and reduced datasets after compression.}
  \begin{center}
    \begin{small}
      \begin{sc}
        \begin{tabular}{lcccc}
          \toprule
          Dataset 
          & Orig.\ samples 
          & Red.\ samples 
          & Orig.\ features 
          & Red.\ features \\
          \midrule
          Titanic
          & 1309 & 1147 & 378 & 367 \\

          skin\_nonskin
          & 245057 & 51444 & 3 & 3 \\

          phishing
          & 11055 & 5849 & 68 & 68 \\

          a7a
          & 16100& 13900& 122& 120\\

          breast-cancer
          & 683 & 675 & 10 & 10 \\
          \bottomrule
        \end{tabular}
      \end{sc}
    \end{small}
  \end{center}
  \vskip -0.1in
\end{table}
To empirically verify that compression preserves the training objective and predictive behavior, we also compare the total objective value and maximum discrepancy in predicted probabilities between the original and lifted reduced solutions.
\begin{table}[H]
  \caption{Objective difference (absolute and relative) and train data prediction difference (absolute) between original and reduced datasets after compression.}
  \begin{center}
    \begin{small}
      \begin{sc}
        \begin{tabular}{lccc}
          \toprule
          Dataset 
          & Absolute $\Delta$ Obj 
          & Relative $\Delta$ Obj
          & Max absolute $\Delta\hat{p}$ \\
          \midrule
          Titanic
          & $4.9659\times 10^{-3}$& $1.3156\times 10^{-3}$\%& $4.7879\times 10^{-3}$ \\

          skin\_nonskin
          & $6.4958\times 10^{-4}$& $1.0764\times 10^{-6}$\%& $1.8252\times 10^{-4}$\\

          phishing
          & $5.3111\times 10^{-2}$& $3.3928\times 10^{-3}$\%& $8.9266\times 10^{-3}$\\

          a7a
          & $3.2449\times 10^{-3}$& $6.2031\times 10^{-5}$\%& $4.4786\times 10^{-3}$\\

          breast-cancer
          & $1.1733\times 10^{-4}$& $2.2807\times 10^{-4}$\%& $1.0820\times 10^{-3}$\\
          \bottomrule
        \end{tabular}
      \end{sc}
    \end{small}
  \end{center}
  \vskip -0.1in
\end{table}

Finally, we compare training runtime on CPU. Runtime is given in seconds and includes time taken for StandardScaler to fit and transform the data matrix.
\begin{table}[H]
  \caption{Effects of compression on training runtime (mean and standard deviation across 50 trials per dataset).}
  \begin{center}
    \begin{small}
      \begin{sc}
        \begin{tabular}{lccc}
          \toprule
          Dataset & Training runtime (original)  & Compression time         & Training runtime (reduced)  \\
          \midrule
          Titanic & 3.8311$\pm$0.8025& 0.0187$\pm$0.0268& 0.5947$\pm$0.2816\\
          skin\_nonskin & 1.2156$\pm$0.3461& 0.1519$\pm$0.0603& 0.1895$\pm$0.0647\\
          phishing & 0.0694$\pm$0.0216& 0.0480$\pm$0.0414& 0.0144$\pm$0.0033\\
          a7a & 0.8996$\pm$0.1889& 0.1404$\pm$0.1323& 0.6836$\pm$0.1674\\
          breast-cancer & 0.3041$\pm$0.2043 & 0.0004$\pm$0.0004& 0.2765$\pm$0.1973\\
          \bottomrule
        \end{tabular}
      \end{sc}
    \end{small}
  \end{center}
  \vskip -0.1in
\end{table}


\end{document}